\documentclass{article}

\usepackage[utf8]{inputenc}
\usepackage{amsmath,amsthm,amssymb,euscript}
\usepackage{graphicx}
\usepackage{wrapfig}

\usepackage{tikz}
\usetikzlibrary{matrix,arrows}

\theoremstyle{plain}
\newtheorem{thm}{Theorem}[section]
\newtheorem{lemma}{Lemma}[section]
\newtheorem{corollary}{Corollary}[section]

\theoremstyle{definition}
\newtheorem{defn}{Definition}[section]
\newtheorem*{defn*}{Definition}

\theoremstyle{remark}
\newtheorem*{remark}{Remark}
\newtheorem*{example}{Example}

\newcommand{\Aut}{\mathrm{Aut\ }}
\newcommand{\Homeo}{\mathrm{Homeo\ }}
\newcommand{\Out}{\mathrm{Out\ }}
\newcommand{\im}{\mathrm{Im}}

\newcommand{\Sub}{\mathrm{L}}
\newcommand{\Co}{\mathrm{C}}
\newcommand{\B}{\mathrm{\bf B}}
\newcommand{\E}{\mathrm{\bf E}}

\begin{document}

\title{Homotopy classification of finite group actions on aspherical spaces}
\author{Lokutsievskiy L.}

\maketitle

\begin{abstract}
    The author proposes a method for investigating actions of finite groups on aspherical spaces. Complete homotopy classification of free actions of finite groups on aspherical spaces is obtained. Also there are some results about non-free actions. For example a relation between the cohomology of finite groups and the lattice structure of its subgroups is obtained by the proposed method. This relation is formulated in terms of spectral sequences.
\end{abstract}

\section{Introduction}

Suppose $G$ is a finite group, $T$ is a Hausdorff topological space, and $p:G\to\Homeo T$ is a free action of $G$ on $T$. Therefore the natural projection $P:T\to T/p(G)$ is a regular covering. Now suppose that $T$ is an Eilenberg-MacLane space of type $K(\pi,1)$; then the long exact sequence of the covering has the form

\begin{equation}
\label{eq_short_seq_pi_G}
    1 \to \pi \to \pi_1\bigl(T/p(G)\bigr) \to G \to 1.
\end{equation}

Thus we have an extension of $G$ by $\pi$. The aim of this paper is to prove that the set of all free actions (up to homotopy conjugation) of a finite group $G$ on aspherical spaces of type $K(\pi,1)$ is in one-to-one correspondence with the set of all extensions of the group $G$ by the group $\pi$ \footnote{This statement is proved under minimal natural restriction: existence of cell complex structure is used}. Classification of such extensions is well known  (see \cite{Brown}) and closely connected to the cohomology of the group $G$.

Also some results about non-free actions are obtained in this paper. For example if $G$ acts (not necessarily freely) on $T$ then the cohomology of the group $G$ and the space $T$ are related by a Hochschield-Mostov spectral sequence. In particular a relation between cohomology of a finite group and the subgroup lattice structure is detected.

\section{Free actions on aspherical spaces}

Let $T$ be an Eilenberg-MacLane space of type $K(\pi,1)$, i.e., $\pi_k(T)=0$ for all $k\neq 1$ and the groups $\pi_1(T)$ and $\pi$ are isomorphic. Sometimes we will use the term ``aspherical space'' instead of ``Eilenberg-MacLane space of type $K(\pi,1)$''.

\begin{remark}
    We will work only with Hausdorff spaces. Throughout what follows the term ``space'' will mean ``Hausdorff topological space''.
\end{remark}

Also let $G$ be a finite group.

\begin{defn}
    Let $p_1$ and $p_2$ be actions of finite groups $G_1$ and $G_2$ on spaces $T_1$ and $T_2$. The action $p_1$ is called homotopy conjugate to $p_2$ ($p_1\sim p_2$) if there exist a homotopy equivalence $\varphi:G_1\to G_2$ and an isomorphism $\theta:G_1\to G_2$ such that

    \begin{equation}
    \label{eq_action_equiv}
        \varphi\circ p_1(g) = p_2\bigl(\theta(g)\bigr)\circ\varphi\mbox{ for all }g\in G_1.
    \end{equation}
\end{defn}


\begin{remark}
    Notice that the classification problem of finite group actions (up to homotopy conjugation) is not well defined in general. It is clear that the relation $\sim$ is reflexive and transitive but it is not symmetric. For example suppose $\mathbb Z_2$ acts on infinite dimensional sphere $S^\infty$ by two ways: trivially $p_1(\pm 1)x=x$ and freely $p_2(\pm 1)x=\pm x$; then $p_2\sim p_1$ by homotopy equivalence $\varphi:S^\infty\to S^\infty$ such that $\varphi:S^\infty\mapsto x_0\in S^\infty$ but there is no map $\varphi':S^\infty\to S^\infty$ such that $\varphi'\circ p_1(\pm 1) = p_2(\pm 1)\circ\varphi'$. However this section is devoted to free actions on aspherical spaces and we will prove that the relation $\sim$ is symmetric in this particular case.
\end{remark}

We will need the following minimal natural restriction for the consequent proofs:

\begin{defn}
    A free action $p$ of a finite group $G$ on a space $T$ is called regular if the quotient space $T/p(G)$ is a cell complex. Note that if $p$ is a regular free action on a space $T$ then $T$ is also a cell complex.
\end{defn}

For example if $G$ acts on an polyhedron \footnote{A polyhedron is a space homeomorphic to the geometric realization of an simplicial complex with weak topology.} $T$ freely and simplicially then $p$ is regular. Of course there are many other examples of regular free actions.

Our aim is to show that the set of all regular free actions on an aspherical space (up to homotopy conjugation) is in one-to-one correspondence with the set of all classes of equivalent extensions of the form (\ref{eq_action_equiv}).

\begin{defn}

    Extensions $1\to\pi_i\to S_i\to G_i\to 1$, $i=1,2$ are called equivalent if they are isomorphic in category of short exact sequences, i.e., there exist isomorphisms $\rho:\pi_1\to\pi_2$, $\xi:S_1\to S_2$, and $\theta:G_1\to G_2$ such that the following diagram is commutative

    \begin{center}
    \begin{tikzpicture}[description/.style={fill=white,inner sep=2pt}]
        \matrix (m) [matrix of math nodes, row sep=1em, column sep=2em, text height=1.5ex, text depth=0.25ex]
        {   1 & \pi_1 & S_1  & G_1 & 1\\
            1 & \pi_2 & S_2  & G_2 & 1 \\ };
        \path[->,font=\scriptsize]
        (m-1-1) edge (m-1-2)
        (m-1-2) edge (m-1-3)
        (m-1-3) edge (m-1-4)
        (m-1-4) edge (m-1-5)
        (m-2-1) edge (m-2-2)
        (m-2-2) edge (m-2-3)
        (m-2-3) edge (m-2-4)
        (m-2-4) edge (m-2-5)
        (m-1-2) edge node[auto] {$ \rho $}  (m-2-2)
        (m-1-3) edge node[auto] {$ \xi $}   (m-2-3)
        (m-1-4) edge node[auto] {$ \theta $}    (m-2-4);
    \end{tikzpicture}
    \end{center}
\end{defn}

\begin{lemma}[well-definiteness]
\label{lm_correctness}
    Let $p_1$ and $p_2$ be free actions of finite groups $G_1$ and $G_2$ on path-connected (not necessarily aspherical) spaces $T_1$ and $T_2$. Suppose there exist a homotopy equivalence $\varphi:T_1\to T_2$ and an isomorphism $\theta:G_1\to G_2$ such that the equation (\ref{eq_action_equiv}) are satisfied; then the short exact sequences

    $$
        1\to\pi_1(T_i)\to\pi_1\bigl(T_i/p_i(G_i)\bigr)\to G_i\to 1,\ i=1,2
    $$

    are equivalent (these sequences are induced by the coverings $T_i\to T_i/p_i(G_i)$).
\end{lemma}

Throughout what follows by $f_\#$ denote the induced homomorphism $f_\#:\pi_k(T_1)\to \pi_k(T_2)$ of homotopy groups where $f:T_1\to T_2$ is a continuous map.

\begin{proof}
    Let us mark a point $x_i^0$ in the space $T_i$ and $y_i^0=P_i(x_i^0)$ in the space $T_i/p_i(G_i)$ where $P_i:T_i\to T_i/p_i(G_i)$ is the natural projection, $i=1,2$. Assume that $x_2^0=\varphi(x_1^0)$.

    Let $\psi:T_1/p_1(G_1)\to T_2/p_2(G_2)$ be the map induced by the map $\varphi$. Since projection $P_1$ is open, we see that the map $\psi$ is continuous. By construction we see that the following diagram is commutative

        \begin{center}
    \begin{tikzpicture}[description/.style={fill=white,inner sep=2pt}]
        \matrix (m) [matrix of math nodes, row sep=1em, column sep=2em, text height=1.5ex, text depth=0.25ex]
        {   1 & \pi_1(T_1,x_1^0) & \pi_1\bigl(T_1/p_1(G_1),y_1^0\bigr)  & G_1 & 1\\
            1 & \pi_1(T_2,x_2^0) & \pi_1\bigl(T_2/p_2(G_2),y_2^0\bigr)  & G_2 & 1 \\ };
        \path[->,font=\scriptsize]
        (m-1-1) edge (m-1-2)
        (m-1-2) edge (m-1-3)
        (m-1-3) edge (m-1-4)
        (m-1-4) edge (m-1-5)
        (m-2-1) edge (m-2-2)
        (m-2-2) edge (m-2-3)
        (m-2-3) edge (m-2-4)
        (m-2-4) edge (m-2-5)
        (m-1-2) edge node[auto] {$ \varphi_\# $}(m-2-2)
        (m-1-3) edge node[auto] {$ \psi_\# $}   (m-2-3)
        (m-1-4) edge node[auto] {$ \theta $}    (m-2-4);
    \end{tikzpicture}
    \end{center}

    It is clear that $\varphi_\#$ is an isomorphism. Using 5-lemma for the diagram we see that the homomorphism $\psi_\#$ is an isomorphism too.
\end{proof}

\begin{corollary}
    If the free action $p_1$ is homotopy conjugate to $p_2$ then the quotient spaces $T_1/p_1(G_1)$ and $T_2/p_2(G_2)$ are weakly homotopy equivalent. Moreover, if the spaces $T_1/p_1(G_1)$ and $T_2/p_2(G_2)$ are cell complexes then from the Whitehead theorem if follows that these spaces are homotopy equivalent.
\end{corollary}

\begin{lemma}[injectiveness]
\label{lm_injection}
    Let $p_1$ and $p_2$ be regular free actions of finite groups $G_1$ and $G_2$ on aspherical spaces $T_1$ and $T_2$. Suppose there exist isomorphisms $\xi:\pi_1\bigl(T_1/p_1(G_1)\bigr)\to \pi_1\bigl(T_2/p_2(G_2)\bigr)$ and $\theta:G_1\to G_2$ such that the following square is commutative (horizontal arrows are the boundary maps)

    \begin{center}
    \begin{tikzpicture}[description/.style={fill=white,inner sep=2pt}]
        \matrix (m) [matrix of math nodes, row sep=1em, column sep=2.5em, text height=1.5ex, text depth=0.25ex]
        {   \pi_1\bigl(T_1/p_1(G_1)\bigr) & G_1 \\
            \pi_1\bigl(T_2/p_2(G_2)\bigr) & G_2 \\ };
        \path[->,font=\scriptsize]
        (m-1-1) edge node[auto] {$ \partial_1 $}    (m-1-2)
        (m-2-1) edge node[auto] {$ \partial_2 $}    (m-2-2)
        (m-1-1) edge node[auto] {$ \xi $}       (m-2-1)
        (m-1-2) edge node[auto] {$ \theta $}        (m-2-2);
    \end{tikzpicture}
    \end{center}

    Then the action $p_1$ is homotopy conjugate to $p_2$ by the isomorphism $\theta$ and a homotopy equivalence $\varphi:T_1\to T_2$ such that $\xi \circ P_{1\#} = P_{2\#} \circ \varphi_\#$.
\end{lemma}

\begin{proof}
    Mark points $x_1^0\in T_1^0$ and $y_1^0=P_1(x_1^0)\in T_1/p_1(G_1)$. By definition, put $y_2^0=\psi(y_1^0)$. Denote by $x_2^0$ any point belonging to the discrete preimage $P_2^{-1}(y_2^0)$.

    Since $\xi$ is isomorphism and CW-complexes $T_i/p_i(G_i)$, $i=1,2$ are aspheric we see that there exists homotopy equivalence $\psi:T_1/p_1(G_1)\to T_2/p_2(G_1)$ such that $\psi(y_1^0)=y_2^0$ and $\psi_\#=\xi$.

    Let us construct the required map $\varphi:T_1\to T_2$: for any point $x\in T_1$ there exists a path $\gamma:[0;1]\to T_1$ joining $x_1^0=\gamma(0)$ to $x=\gamma(1)$. Consider the path $\psi\circ P_1\circ\gamma$ in the space $T_2/p_2(G_2)$ joining $y_2^0$ to $\psi(P_1(x))$. There exists a unique lifting $\gamma'$ of this path to the space $T_2$ such that $\gamma'(0)=x_2^0$. So $P_2 \circ \gamma'  =  \psi \circ P_1 \circ \gamma$. By definition, put $\varphi(x)=\gamma'(1)$.

    First we shall prove that the map $\varphi$ is well defined. Obviously, it is sufficient to prove that if $\gamma$ is a loop in $T_1$ ($x=x_1^0$) then $\varphi(x_1^0)=x_2^0$. Let $[\gamma]\in\pi_1(T_1,x_1^0)$ be the homotopy class of the loop $\gamma$; then $P_{1\#}([\gamma])$ belongs to the kernel $\ker\partial_1$. Since the diagram from the condition is commutative, we have $\psi_\#(P_{1\#}([\gamma]))\in\ker\partial_2$. Since covering $T_2\to T_2/p_2(G_2)$ is regular and $P_2\circ\gamma' = \psi\circ P_1\circ\gamma$, we see that $\gamma'$ is a loop in $T_2$ as was to be proved.

    Notice that the following diagram is commutative by construction of $\varphi$:

    \begin{center}
    \begin{tikzpicture}[description/.style={fill=white,inner sep=2pt}]
        \matrix (m) [matrix of math nodes, row sep=1.5em, column sep=2.5em, text height=1.5ex, text depth=0.25ex]
        {   T_1     & T_2       \\
            T_1/p_1(G_1)    & T_2/p_2(G_2)  \\
        };
        \path[->,font=\scriptsize]
            (m-1-1) edge node[auto] {$ \varphi $}   (m-1-2)
            (m-2-1) edge node[auto] {$ \psi $}  (m-2-2)
            (m-1-1) edge node[auto] {$ P_1 $}   (m-2-1)
            (m-1-2) edge node[auto] {$ P_2 $}   (m-2-2);
    \end{tikzpicture}
    \end{center}

    It is important to note that if $\gamma$ is a path in $T_2$ then the path $\varphi\circ\gamma$ in $T_2$ is the same as the lift $\gamma'$ of the path $\psi\circ P_1\circ \gamma$ into the space $T_2$ such that $\gamma'(0)=\varphi(\gamma(0))$.

    \begin{figure}
    \label{fig_phi_continuty}
        \includegraphics[width=13.5cm]{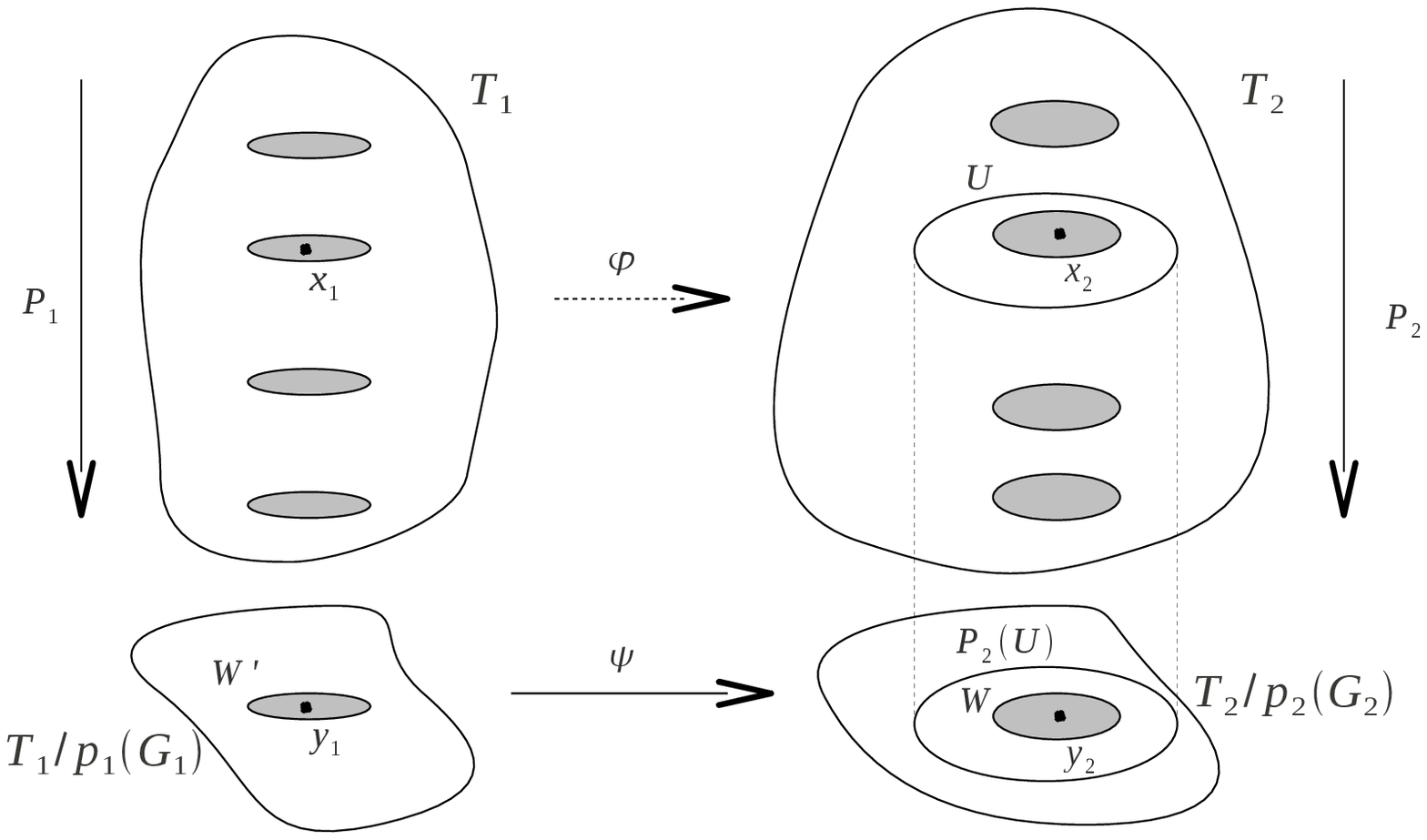}
    \end{figure}

    Secondly we shall show that the constructed map $\varphi$ is continuous. Let $x_1$ be a point in $T_1$, $x_2=\varphi(x_1)\in T_2$, $U$ a neighborhood of the point $x_2$ in $T_2$, and $y_i=P_i(x_i)$ where $i=1,2$. It is sufficient to find a neighborhood of the point $x_1$ such that $\varphi(V)\subseteq U$.

    Since the projection $P_2$ is open, we see that $P_2(U)$ is open too. Hence there exists a neighborhood $W\subseteq P_2(U)$ of the point $y_2$ such that $P_2^{-1}(W)$ is homeomorphic to the direct product $P_2^{-1}(y_2)\times W$. The set $x_2\times W$ is open in $T_2$ because the fibre $P_2^{-1}(y_2)$ is discrete.

    The sets $(x_2\times W)\cap U$ and $P_2\bigl((x_2\times W)\cap U\bigr)$ are open. Whence without loss of generality it can be assumed that $x_2\times W$ is contained in $U$ (we can replace $W$ by $P_2\bigl((x_2\times W)\cap U\bigr)$ if necessary).

    It follows from continuity of $\psi$ that $\psi^{-1}(W)$ is open in $T_1/p_1(G)$. Consequently there exists a neighborhood $W'\subseteq\psi^{-1}(W)$ of the point $y_2$ such that the preimage $P_1^{-1}(W')$ is homeomorphic to $P_1^{-1}(y_1)\times W'$. Since the cell complex $T_1/p_1(G_1)$ is locally contractible (see \cite{Hatcher}), we may assume that the neighborhoods $W'$ and $x_1\times W'$ are path-connected.

    By $P_2\bigl(\varphi(x_1\times W')\bigr) = \psi(W') \subseteq W$ so we see that $\varphi(x_1\times W')$ is locally connected too. Therefore, we see that $\varphi(x_1\times W')$ belongs to the level $x_2\times W\subseteq U$. This means that $\varphi$ is continuous as was to be proved.

    Thirdly let us prove that $\varphi$ is a homotopy equivalence of $T_1$ and $T_2$. From the Whitehead theorem it follows that $\varphi$ is a homotopy equivalence iff $\varphi_\#$ is an isomorphism of the groups $\pi_1(T_1)$ and $T_2$. The following diagram is commutative:

        \begin{center}
    \begin{tikzpicture}[description/.style={fill=white,inner sep=2pt}]
        \matrix (m) [matrix of math nodes, row sep=1.5em, column sep=2.5em, text height=1.5ex, text depth=0.25ex]
        {   1   & \pi_1(T_1)    & \pi_1\bigl(T_1/p_1(G_1)\bigr) & G_1   \\
            1   & \pi_1(T_2)    & \pi_1\bigl(T_2/p_2(G_2)\bigr) & G_2   \\
        };
        \path[->,font=\scriptsize]
            (m-1-1) edge                (m-1-2)
            (m-1-2) edge node[auto] {$P_{1\#}$} (m-1-3)
            (m-1-3) edge                (m-1-4)

            (m-2-1) edge                (m-2-2)
            (m-2-2) edge node[auto] {$P_{2\#}$} (m-2-3)
            (m-2-3) edge                (m-2-4)

            (m-1-2) edge node[auto] {$ \varphi_\# $}    (m-2-2)
            (m-1-3) edge node[auto] {$ \psi_\# $}       (m-2-3)
            (m-1-4) edge node[auto] {$ \theta $}        (m-2-4);
    \end{tikzpicture}
    \end{center}

    Using 5-lemma we see that $\varphi_\#$ is an isomorphism.

    Finally we shall prove that the map $\varphi$ is equivariant\footnote{The term ``equivariant'' is usually used for actions of one group. Here we mean ``equivariant with the respect to the isomorphism $\theta$''.}. For any $g\in G_1$  there exists a path $\gamma_g$ connecting points $x_1^0$ and $p_1(g)x_1^0$. Let $x$ be a point in $T_1$, $\gamma$ be a path connecting $x_1^0$ and $x$ and $\omega=\gamma_g\oplus\bigl(p_1(g)\circ\gamma\bigr)$ be the path connecting $x_1^0$ and $p_1(g)x$ (by definition, $\omega(0)=\gamma_g(0)=x_1^0$, $\omega(\frac{1}{2})=\gamma_g(1)=p_1(g)x_1^0$ and $\omega(1)=p_1(g)\gamma(1)=p_1(g)x$).

    Let us construct an explicit lift of the path $\psi\circ P_1\circ\omega$ into the space $T_2$. Note that since $\psi\circ P_1\circ\gamma_g$ is a loop, we get $P_2\Bigl(\varphi\bigl(p_1(g)x_1^0\bigr)\Bigl) = y_2^0$. Hence there exists an element $g'\in G_2$ such that $\varphi\bigl(p_1(g)x_1^0\bigr) = p_2(g')x_2^0$. We obtain $g'=\theta(g)$ by commutativity of the diagram from the condition.

     Consider the lift $\gamma'_g$ of the path $\psi\circ P_1\circ\gamma_g$ such that $\gamma'_g(0)=x_2^0$. By the above we get $\gamma'_g(1) = p_2\bigl(\theta(g)\bigr)x_2^0$. Thus if $\gamma'$ is the lift of the path $\psi \circ P_1\circ \gamma$ such that $\gamma'(0)=x_2^0$; then the path  $\omega'=\gamma_g'\oplus\bigl(p_2(\theta(g))\circ\gamma'\bigr)$ is the lift of the path $\psi\circ P_1\circ\omega$ such that $\omega'(0)=x_2^0$. Consequently, $\varphi\bigl(p_1(g)x\bigr) = \omega'(1) = p_2\bigl(\theta(g)\bigr)\gamma'(1) = p_2\bigl(\theta(g)\bigr)x$. Q.E.D.
\end{proof}

\begin{remark}
    If the action $p_1$ is homotopy conjugate to $p_2$ then the spaces $T_1$ and $T_2$ are homotopy equivalent.
\end{remark}

\begin{corollary}[symmetry of $\sim$]
    If free actions $p_1$ and $p_2$ are regular then $p_1\sim p_2$ iff $p_2\sim p_1$.
\end{corollary}

\begin{proof}
    Suppose, $p_1\sim p_2$. Using lemma \ref{lm_correctness} we see that there exists an isomorphism $\xi:\pi_1\bigl(T_1/p_1(G)\bigr) \to \pi_1\bigl(T_2/p_2(G)\bigr)$ such that the appropriate diagram is commutative. Using the previous lemma to the isomorphisms $\xi^{-1}$ and $\theta^{-1}$ we immediately get $p_2\sim p_1$
\end{proof}

\begin{lemma}[surjectiveness]
    Let $1\to\pi\to S\to G\to 1$ be an extension of a finite group $G$ by a discrete group $\pi$; then there exist an aspheric space $T$ of type $K(\pi,1)$ and a regular action $p$ of $G$ on $T$ such that the following short sequence is equivalent to the first one:

    $$
        1 \to \pi_1(T) \to \pi_1\bigl(T/p(G)\bigr) \to G \to 1
    $$
\end{lemma}

\begin{proof}
    Let us use the classic the Milnor construction (see \cite{Feigen_Fuks,Husemoller}). Consider the group $S$ as a discrete topological space; then the space $\E S = S * S * S * \ldots$ is a join of countable number of the $S$ copies. The space $\E S$ consists of points

    $$
        \bigl<t,x\bigr> = (t_1 x_1,t_2 x_2, \ldots, t_n x_n,\ldots),
    $$

    where $x_i\in S$ and $t_i\in [0;1]$. Moreover, only finite number of $t_i$ is non-zero. Points $\bigl<t,x\bigr>$ and $\bigl<t,x\bigr>$ are considered identical if $t_i=t'_i$ for all $i$ and $x_i=x'_i$ when $t_i=t'_i > 0$. By definition the topology on $\E S$ is the weakest topology such that the following maps are continuous:

    $$
        \begin{array}{ll}
            T_i: \E G \to [0;1]; \quad& T_i\left(\bigl<t,x\bigr>\right) = t_i;\\
            X_i: T_i^{-1}(0;1] \to S; \quad& X_i\left(\bigl<t,x\bigr>\right) = x_i;\\
        \end{array}
    $$

    It is obvious that the space $\E S$ is an simplicial complex.

    The left action of $S$ on $\E S$ is free, continuous and simplicial. We claim that this action is discrete, i.e. for any point $\bigl<t,x\bigr>$ there exists a neighborhood $U\ni\bigl<t,x\bigr>$ such that the sets $gU$, $g\in G$ are mutually disjoint. Indeed,

    $$
        U=\bigcap_{i:t_i\neq 0} \bigl(T_i^{-1}(0;1]\cap X_i^{-1}(x_i)\bigr).
    $$

    Since $U$ is an intersection of finite number of open sets; then $U$ is open. It is obvious that the sets $gU$ are mutually disjoint.

    Thus the group $\pi$ as an subgroup of $S$ acts freely and discretely on $\E S$. Consequently $\E S\to T=\E S/\pi$ is a regular covering. Obviously $T$ is an Eilenberg-MacLane space of type $K(\pi,1)$. Since action of $\pi$ on $\E S$ is simplicial, we see that $T$ has an simplicial structure (see \cite{Bredon}).

    Clearly there exists a continuous and free action of $G=S/\pi$ on $T$ and $T/G=\E S/S$.
\end{proof}

Finally let us note that we can not weaken the conditions of the previous lemma, i.e., we can not require that for any aspheric space $T$ there exists a necessary free action of the group $G$. For example let $T$ be a wedge product of two circles. Since any homeomorphism leaves fixed the sewing point, we see that for any non-trivial group there is no its free action on $T$.

\begin{thm}[classification]
\label{thm_class_free_action}
    Let $G$ be a finite group and $\pi$ a discrete group. Then the set of all regular free actions (up to homotopy conjugation) of $G$ on Eilenberg-MacLane spaces of type $K(\pi,1)$ is in one-to-one correspondence with the set of all classes of equivalent extensions of $G$ by $\pi$. Here if $p$ is a regular free action on an aspheric space $T$ then the following extension corresponds to the action $p$:

    $$
        1\to \pi_1(T) \to \pi_1\bigl(T/p(G)\bigr) \to G \to 1.
    $$
\end{thm}

Let us consider a simple illustrative example:

\begin{example}
    Let the group $G=\mathbb Z_2$ act on $T=S_1\times S^\infty$. Since $S^\infty$ is contractible, then $T$ is an aspherical space of type $K(\mathbb Z,1)$ . There exists only three different extensions of the group $\mathbb Z_2$ by $\mathbb Z$:

    \begin{enumerate}
        \item The group $\mathbb Z_2$ acts on $S^1$ trivially. So, $p_1(\pm 1)(x,y)=(x,\pm y)$ where $(x,y)\in S^1\times S^\infty$ Then we get $(S^1\times S^\infty)/p_1(\mathbb Z_2) = S^1\times\mathbb RP^\infty$. Thus we have the first extension:

        $$
            1\to\mathbb Z \to \mathbb Z \times \mathbb Z_2 \to \mathbb Z_2 \to 1.
        $$

        \item The group $\mathbb Z_2$ acts on $S^1$ by reflection with respect to the center. So, $p_2(\pm 1)(x,y)=(\pm x, y)$. Since $(S^1\times S^\infty)/p_2(\mathbb Z_2)=S_1\times S^\infty$, then we obtain the second extension

        $$
            1\to\mathbb Z  \xrightarrow{\times 2}  \mathbb Z \to \mathbb Z_2 \to 1.
        $$

        \item The group $\mathbb Z_2$ acts on $S^1\subseteq\mathbb C$ by reflection with respect to a diameter. So, $p_3(\pm 1)(x,y)=(x^{\pm 1}, \pm y)$. Hence we get the last one

        $$
            1\to\mathbb Z \to \mathbb Z \rtimes_{(-1)} \mathbb Z_2 \to \mathbb Z_2\to 1.
        $$

         Here we have to a little bit look ahead. Now we have no effective methods for calculating of this extension. But we will prove lemma \ref{lm_semidirect_product} and the extension can be easily calculated by this lemma.
    \end{enumerate}
\end{example}

\section{Congruent extensions}
Classification of classes of congruent extensions is well known (Eulenberg, MacLane, 1947, see \cite{EilenbergMacLane}).

\begin{defn}
    Equivalent extensions $1\to\pi_i\to S_i\to G_i\to 1$, $i=1,2$ are called congruent if $\pi_1=\pi_2=\pi$, $G_1=G_2=G$ and $\rho=id_\pi$, $\theta=id_G$

        \begin{center}\begin{tikzpicture}
        \matrix (m) [matrix of math nodes, row sep=0.4em, column sep=1.5em, text height=1.5ex, text depth=0.25ex]
        {
                &   & S_1   &   &   \\
            1   & \pi   &   & G & 1 \\
                &   & S_2   &   &   \\
        };
        \path[->,font=\scriptsize]
        (m-2-1) edge (m-2-2)
        (m-2-2) edge (m-1-3)
            edge (m-3-3)
        (m-1-3) edge (m-2-4)
            edge node[auto] {$\xi$} (m-3-3)
        (m-3-3) edge (m-2-4)
        (m-2-4) edge (m-2-5);
    \end{tikzpicture}\end{center}
\end{defn}

We can easy reformulate constructed homotopy classification in theorem \ref{thm_class_free_action} for the congruent extensions language:

\begin{defn}
    An aspheric space $T$ is called an Eilenberg-MacLane space of type $K(\pi,1)$ in strong sense if it is marked by a point $x^0\in T$ and an isomorphism $\rho:\pi_1(T,x^0)\to \pi$ is fixed.
\end{defn}

Since the isomorphism $\rho$ is fixed, we can use $\pi$ instead of $\pi_1(T,x^0)$.

\begin{thm}[classification]
\label{thm_strong_class_free_action}
    Let $G$ be a finite group and $\pi$ a discrete group. Then the set of all regular free actions (up to strong homotopy conjugation) of $G$ on Eilenberg-MacLane spaces of type $K(\pi,1)$ in strong sense is in one-to-one correspondence with the set of all classes of congruent extensions of $G$ by $\pi$. Here if $p$ is a regular free action of $G$ on an aspheric space $T$ of type $K(\pi,1)$ in strong sense then the following extension corresponds to the action $p$:

    $$
        1\to \pi_1(T) \to \pi_1\bigl(T/p(G)\bigr) \to G \to 1.
    $$
\end{thm}

A proof of this theorem is similar to the proof of the theorem \ref{thm_class_free_action} and based on lemmas almost word for word repeated the lemmas on well-definiteness, injectiveness and sujectiveness.

\begin{example}
    If the group $\pi$ is abelian then the action $p$ of $G$ on $T$ induces the action $\eta$ of $G$ on $\pi$ ($\eta:G\to \Aut\pi$). $G$-module structure on $\pi$ is defined by the map $\eta$. The set of all congruence classes of extensions of $G$ by $\pi$ is exactly the second cohomology $H^2(G,\pi)$ (the structure of $G$-module on $\pi$ is fixed). If $G$ is non-abelian then there exists only homomorphism $\eta:G\to \Out \pi$ of $G$ into the outer automorphism class group of the group $\pi$. In this case the set of all congruence classes of extensions of $G$ by $\pi$ is the second cohomology $H^2(G,C(\pi))$ ($C(\pi)$ is the center of $\pi$) or empty. This depends on an obstruction in the third cohomology $H^3(G,C(\pi))$.
\end{example}

The results of the classification theorems \ref{thm_class_free_action} and \ref{thm_strong_class_free_action} can be generalize to infinite discrete groups. However in this case the projection $T\to T/G$ may not be a covering. Since we essentially use the covering structure on $T\to T/G$, then one have to require essentially stronger conditions than we used.

\section{Non-free actions}

Let us remember the $S^\infty$ example from the first paragraph. We see that there are some difficulties with the notion of homotopy conjugation for non-free actions. So let us use the Borel construction.

\begin{defn}
    Let $p$ be an action (optionally non-free) of a finite group $G$ on an aspherical space $T$. The diagonal free action $p_f$ is induced on the space $T\times \E G$ where $\E G$ is a contractible space and $G$ acts regularly and freely on it. The following extension is called a subordinate extension to the action $p$:

    \begin{equation}
    \label{eq_slave_extension}
        1\to\pi_1(T)\to\pi_1(T\times\E G/G)\to G\to 1.
    \end{equation}
\end{defn}

Now we show that the extension (\ref{eq_slave_extension}) is well defined. Thus we have to prove that this extension is independent of the choice of the space $\E G$. Let $\E G$ and $\E G'$ be contractible spaces and let $G$ act regularly and freely on them. Since $\E G$ and $\E G'$ are contractible, then from theorem \ref{thm_strong_class_free_action} it follows that the actions $p_f$ and $p_f'$ are strongly homotopy conjugate. Consequently there exists an equivivariant homotopy equivalence $\psi:\E G\to \E G'$. By $\varphi$ denote the map $\varphi:T\times\E G\to T\times\E G'$, $\varphi(x,e) = (x,\psi(e))$. Obviously, $\varphi$ is an equivariant homotopy equivalence. So from lemma \ref{lm_correctness} it follows that the induced extensions are isomorphic. Moreover, if we fix the natural isomorphisms $\pi_1(T\times \E G) = \pi_1(T)$ and $\pi_1(T\times \E G') = \pi_1(T)$, then these extensions will be congruent.

\begin{lemma}
    If the action $p$ is free, then the subordinate extension (\ref{eq_slave_extension}) and the extension from the classification theorem \ref{thm_class_free_action} are congruent.
\end{lemma}

\begin{proof}
    It is sufficient to prove that the actions $p$ and $p_f$ are strongly homotopy conjugate. The projection $T\times \E G\to T$ is the necessary equivariant homotopy equivalence.
\end{proof}

It is easily shown that

\begin{lemma}
    Let $p_1$ and $p_2$ be an actions (optionally non-free) of finite groups $G_1$ and $G_2$ on aspherical spaces $T_1$ and $T_2$. Suppose there exists a homotopy equivalence $\varphi:T_1:\to T_2$ and an isomorphism $\theta:G_1\to G_2$ such that the equation (\ref{eq_action_equiv}) holds. Then the subordinate extensions are equivalent.
\end{lemma}

\begin{proof}
    From lemma \ref{lm_injection} it follows that the actions of $G_1$ on $\E G_1$ and $G_2$ on $\E G_2$ are homotopy conjugate. Consequently there exists a homotopy equivalence $\Phi:\E G_1\to \E G_2$ such that $\Phi(gx)=\theta(g)\Phi(x)$ for all $g\in G_1$ and $x\in T_1$.

    Consider the homotopy equivalence $\varphi':T_1\times \E G_1 \to T_2\times \E G_2$, $\varphi':(x,e)\mapsto (\varphi(x),\Phi(e))$. Using lemma \ref{lm_correctness}, we see that the extension induced by the diagonal actions of $G_1$ on $T_1\times \E G_1$ and $G_2$ on $T_2\times \E G_2$ are equivalent.
\end{proof}

So the subordinate extension (\ref{eq_slave_extension}) allow us to distinguish actions of a finite group on aspherical spaces.

\section{Calculation methods of subordinate extensions}

\begin{lemma}[about a fixed point]
\label{lm_semidirect_product}
    If an action $p$ of a finite group $G$ on an aspherical space $T$ has a fixed point, then the subordinate extension is decomposed, i.e., it is a semidirect product.
\end{lemma}

\begin{proof}
    There is a Cartan-Borel fibration $P:(T\times\E G)/G\to \E G/G$ associated with the action $p$ where $P:G(x,e)\mapsto Ge$ for $x\in T$ and $e\in\E G$. The map $P$ is induced by the equivariant projection $T\times\E G\to \E G$ and since $G$ is finite, then $P$ is locally trivial fibration. The fibre of this fibration is $T$.

    It is clear that the exact sequence of this fibration

    \begin{equation}
    \label{eq_exact_seq_Cortan_Borel}
        1 \to \pi_1(T) \to \pi_1\bigl((T\times\E G)/G\bigr) \xrightarrow{P_\#} \pi_1(\E G/G) \to 1
    \end{equation}

    is equivalent to the subordinate extension (\ref{eq_slave_extension}).

    Let $x_0\in T$ be a fixed point of the action $p$. By definition, put $S:\E G/G\to(T\times\E G)/G$. $S$ is a section of $P$ induced by the equivariant embedding $\E G\to T\times\E G$, $e\mapsto (x_0,e)$. Consider the maps $S_\#$ and $P_\#$ of the fundamental groups ($S_\#$ and $P_\#$ are induced by $S$ and $P$). Since $P\circ S=id$, then the homomorphism $S_\#$ is a right inverse for $P_\#$. Q.E.D.
\end{proof}

Note that the structure of the semidirect product on $\pi_1\bigl((T\times\E G)/G\bigr)$ is defined by the induced action of $G$ on $\pi_1(T,x_0)$. This action is well defined because the point $x_0$ is fixed.

Also, relations between extensions subordinated to a group and its subgroup can help in some calculations.

\begin{lemma}
\label{lm_subgroups}
    Let a finite group $G$ act on an aspherical space $T$ and let $H$ be its subgroup. Consider the induced action of $H$ on $T$. Then the following diagram is commutative and has exact rows and columns\footnote{If $H$ is not a normal subgroup of $G$, then the third row is considered as an exact sequence of pointed sets.}:

        \begin{center}\begin{tikzpicture}
        \matrix (m) [matrix of math nodes, row sep=1.5em, column sep=1.5em, text height=2.5ex, text depth=0.25ex]
        {
                & 1     & 1                 & 1     &   \\
            1   & \pi_1(T)  & \pi_1\bigl((T\times\E H)/H\bigr)  & H     & 1 \\
            1   & \pi_1(T)  & \pi_1\bigl((T\times\E G)/G\bigr)  & G     & 1 \\
                & 1     & \pi_1((T\times\E G)/G)\Big/
                            \pi_1((T\times\E H)/H)      & G/H       & 1 \\
                &       & 1                 & 1     &   \\
        };
        \path[->,font=\scriptsize]
            (m-2-1) edge                (m-2-2)
            (m-2-2) edge node[auto] {$P^H_\#$}  (m-2-3)
            (m-2-3) edge node[auto] {$\partial$}    (m-2-4)
            (m-2-4) edge                (m-2-5)

            (m-3-1) edge                (m-3-2)
            (m-3-2) edge node[auto] {$P^G_\#$}  (m-3-3)
            (m-3-3) edge node[auto] {$\partial$}    (m-3-4)
            (m-3-4) edge                (m-3-5)

            (m-4-2) edge                (m-4-3)
            (m-4-3) edge node[auto] {$j_3$}     (m-4-4)
            (m-4-4) edge                (m-4-5)

            (m-1-2) edge                (m-2-2)
            (m-2-2) edge node[auto] {$id$}      (m-3-2)
            (m-3-2) edge                (m-4-2)

            (m-1-3) edge                (m-2-3)
            (m-2-3) edge node[auto] {$P_\#$}    (m-3-3)
            (m-3-3) edge node[auto] {$j_1$}     (m-4-3)
            (m-4-3) edge                (m-5-3)

            (m-1-4) edge                (m-2-4)
            (m-2-4) edge                (m-3-4)
            (m-3-4) edge node[auto] {$j_2$}     (m-4-4)
            (m-4-4) edge                (m-5-4);
    \end{tikzpicture}\end{center}

    where $P^G:T\to T/G$, $P^H:T\to T/H$, and $P:T/H\to T/G$ are the natural projections.
\end{lemma}

\begin{proof}
    Without loss of generality we can assume that the action $p$ is free (the Borel construction allows us to instantly rewrite the following proof for the non-free case).

    Consider the following commutative diagram

    \begin{center}\begin{tikzpicture}[description/.style={fill=white,inner sep=2pt}]
        \matrix (m) [matrix of math nodes, row sep=3em, column sep=2.5em, text height=1.5ex, text depth=0.25ex]
        {
            T   &   & T/H   \\
                & T/G   &   \\
        };
        \path[->,font=\scriptsize]
            (m-1-1) edge node[auto] {$P^G$} (m-2-2)
                edge node[auto] {$P^H$} (m-1-3)
            (m-2-2) edge node[auto] {$P$}   (m-1-3);
    \end{tikzpicture}\end{center}

    Note that $P^G:T\to T/G$ and $P^H:T\to T/H$ are regular coverings with finite fibres. $P:T/H\to T/G$ is a covering and it is regular iff $H$ is a normal subgroup of $G$.

    Since the previous diagram is commutative, then the diagram from the conditions is commutative too. The first and the second rows are exact as they are parts of the long exact sequences of the coverings $P^H:T\to T/H$ and $P^G:T\to T/G$. The first and the third columns are obviously exact.

    Let us show that the homomorphism $P_\#$ is a monomorphism: the map $H\to G$ is a monomorphism. Consequently, $\ker P_\# \subseteq \ker\partial=\im P^H_\#$. Therefore, $\ker P_\#\subseteq P^H_\#\bigl( \ker(P_\#\circ P^H_\#) \bigr)$. On the over hand, $P_\#\circ P^H_\# = P^G_\#$ is a monomorphism. Thus, $\ker P_\# \subseteq P^H_\#\bigl(\ker(P^G_\#)\bigr) = \{e\}$. That is the second column is exact.

    Finally, we shall prove that the last row is also exact. Let $j_1:\pi_1(T/G) \to \pi_1(T/G)\big/\pi_1(T/H)$, $j_2:G \to G/H$, and $j_3:\pi_1(T/G)\big/\pi_1(T/H) \to G/H$ be the homomorphisms from the initial diagram. First we have to describe $j_3$: suppose $x\in \pi_1(T/G)\big/\pi_1(T/H)$, then there exists $y\in \pi_1(T/G)$ such that $j_1(y)=x$ as $j_1$ is an epimorphism. Then we have $j_3(x) = (j_2\circ\partial) (y)$. We see that $j_3(x)$ is independent of the choice of the element $y$ because all elements from the coset $y P_\#\bigl( \pi_1(T/H) \bigr)$ map to the same element:

    $$
        (j_2\circ\partial)\Bigl( y P_\#\bigl( \pi_1(T/H) \bigr) \Bigr) =
        (j_2\circ\partial) (y) \cdot (j_2\circ\partial)\Bigl( P_\#\bigl( \pi_1(T/H) \bigr) \Bigr) =
        (j_2\circ\partial) (y).
    $$

    \begin{enumerate}
        \item Surjectiveness of $j_3$: for any $x\in G/H$ there exists $y\in\pi_1(T/G)$ such that $(j_2\circ\partial)(y)=x$ ($j_2$ and $\partial$ are epimorphisms). Hence $j_1(y)\in \pi_1(T/G)\big/\pi_1(T/H)$ is a preimage of $x$.
        \item Injectiveness of $j_3$: suppose $x\in\pi_1(T/G)\big/\pi_1(T/H)$ and $j_3(x)=1$. Then there exists $y\in \pi_1(T/G)$ such that $j_1(y)=x$. Obviously, $\partial y\in \ker j_2=H$. Hence there exists an element $z\in\pi_1(T/H)$ such that $\partial y=\partial z=\partial P_\#(z)$. Since $\partial\bigl( y P_\#(z^{-1}) \bigr) = 1\in G$; then there exists an element $t\in \pi_1(T)$ such that $P^G_\#(t) = y P_\#(z^{-1})$. Thus we get $P_\#\bigl( P^H_\#(t)z \bigr) = P^G_\#(t) P_\#(z) = y$, i.e., $y\in\im P_\#$. Consequently, $x=j_1(y)=1$, Q.E.D.
    \end{enumerate}
\end{proof}

\begin{corollary}
    $\pi_1\bigl((T\times\E H)/H\bigr) = \partial^{-1}(H)$ is a subgroup in $\pi_1\bigl((T\times\E G)/G\bigr)$ and its index is the same as the index of $H$ in $G$. Moreover, $\partial^{-1}(H)$ is a normal subgroup of $\pi_1\bigl((T\times\E G)/G\bigr)$ iff $H$ is a normal subgroup of $G$.
\end{corollary}

\begin{remark}
    If the action $p$ of $G$ on $T$ is free and regular, then the induced action of the subgroup $H\subseteq G$ is also regular (the space $T/p(H)$ is a cell complex because the natural projection $T/p(H)\to T/p(G)$ is a covering and the space $T/p(G)$ is a cell complex).
\end{remark}

The two previous lemmas are very useful because sometimes we can avoid explicit awkward calculations of the group $\pi_1\bigl((T\times\E G)/G\bigr)$.

\begin{example}
    Since any connected graph is an aspherical space; then any action of a finite group on it induces a subordinate extension.

    Let the dihedral group $G=D_{2n}$, $n\geq 3$ act on a circle $T=S^1$ (in this case the circle $S^1$ is an n-gon and the group $D_{2n}$ acts on it by vertex permutation). The normal subgroup $\mathbb Z_n\subseteq D_{2n}$ acts freely and regularly on it and it is clear that $S^1/\mathbb Z_n=S^1$. Hence from lemma \ref{lm_subgroups} it follows that the following diagram is commutative and its columns and rows are exact:

    $$
        \begin{array}{ccccccccc}
            &&1&&1&&1&&\\
            &&\downarrow&&\downarrow&&\downarrow&&\\
            1&\to&\mathbb Z&\xrightarrow{\times n}&\mathbb Z&\to&\mathbb Z_n&\to&1\\
            &&\downarrow&&\downarrow&&\downarrow&&\\
            1&\to&\mathbb Z&\to&A&\to&D_{2n}&\to&1\\
            &&\downarrow&&\downarrow&&\downarrow&&\\
            &&1&\to&\mathbb Z_2&\to&\mathbb Z_2&\to&1\\
            &&&&\downarrow&&\downarrow&&\\
            &&&&1&&1&&\\
        \end{array}
    $$

    Thus the subgroup $A$ is an extension of $\mathbb Z_2$ by $\mathbb Z$. Clearly, the only one variant is possible: $A=\mathbb Z\rtimes \mathbb Z_2$. So the extension subordinated to the natural action of $D_{2n}$ on $S^1$ is

    $$
        1 \to \mathbb Z \xrightarrow{\times n} \mathbb Z \rtimes_{(-1)} \mathbb Z_2 \to D_{2n} \to 1.
    $$

\end{example}

\section{Spectral sequences}

For any action $p$ of a finite group $G$ on a space $T$ there exists a Cartan-Serre spectral sequence for the diagonal action on $T\times\E G$ (see \cite{Bredon_Puchki}):

\begin{equation}
\label{eq_spectral_Cortan_Serre}
    \begin{array}{lllll}
        E^2_{p,q} & = & H_p\bigl(G,H_q(T)\bigr) & \Rightarrow & H_{p+q} \bigl((T\times\E G)/G\bigr);\\
        E_2^{p,q} & = & H^p\bigl(G,H^q(T)\bigr) & \Rightarrow & H^{p+q} \bigl((T\times\E G)/G\bigr).\\
    \end{array}
\end{equation}

Here the abelian groups $H_*(T)$ and $H^*(T)$ are $G$-modules induced by the action $p$.

Let $T$ and $(T\times\E G)/G$ be cell complexes. So if the space $T$ is aspherical, then $H^*(T)$ and $H^*\bigl((T\times\E G)/G\bigr)$ are a group cohomology and the spectral sequence (\ref{eq_spectral_Cortan_Serre}) is the same as the spectral sequence of Hochschild-Mostov for the extension (\ref{eq_slave_extension}) subordinated by $p$:

$$
    \begin{array}{lllll}
        E^2_{p,q} & = & H_p\Bigl(G,H_q\bigl(\pi_1(T)\bigr)\Bigr) & \Rightarrow &
            H_{p+q} \Bigl(\pi_1\bigl((T\times\E G)/G\bigr)\Bigr);\\
        E_2^{p,q} & = & H^p\Bigl(G,H^q\bigl(\pi_1(T)\bigr)\Bigr) & \Rightarrow &
            H^{p+q} \Bigl(\pi_1\bigl((T\times\E G)/G\bigr)\Bigr).\\
    \end{array}
$$

Thus there is an information about relations between cohomology of the group $G$ and the space $T$ in the subordinated extension (\ref{eq_slave_extension}).

\section{Group lattices}
Theory of classifying spaces for small categories is well-known. It was developed in the last century by Segal and Quillen (see \cite{Quillen_CS,Segal_CS}). Main ideas of these works previously appeared in Grothedick's paper (see \cite{Grothendieck}). Some good results can be obtained on this field.

Let $P$ be a poset. Consider $P$ as a small category in natural sense. By $\B P$ denote the classifying space of the small category $P$. The space $\B P$ is a simlplicial complex (with weak topology). Vertexes of $\B P$ are points in $P$, segments are inequalities $a<b$, triangles are inequalities $a<b<c$, etc.

Note that (co)homology $H_*(P)$ ($H^*(P)$), Eulerian haracteristic, etc. of the poset $P$ is exactly (co)homology $H_*(\B P)$ ($H^*(\B P)$), Eulerian haracteristic, etc. of the classifying space $\B P$.

Let a finite group $G$ act on $P$ with respect to the order on $P$. Then this action naturally induce an action on the space $\B P$.

For any finite group $G$ consider its lattices:

\begin{defn}
    Let $\Co G$ be the coset poset of $G$ (without $\hat 0 = \emptyset$ and $\hat 1=G$) with respect to the set-theoretic inclusion.
\end{defn}

The poset of all cosets between $g_1H_1$ and $g_2H_2$ is called the segment between $g_1H_1$ and $g_2H_2$ and is denoted by $[g_1H_1,g_2H_2]$:

$$
    [g_1 H_1,g_2 H_2] = \{gH\in\Co G: g_1 H_1 \subsetneq gH \subsetneq g_2 H_2\}.
$$

Without loss of generality we can consider only segments of the form $[H_1,H_2]$ where $H_i$ is a subgroup of $G$ or $\emptyset$ (since any segment $[g_1H_1,g_2H_2]$ is isomorphic to the segment $[H_1,H_2]$).

\begin{defn}
    The segment $[e,G]$ is called The subgroup lattice and is denoted by $\Sub G$.
\end{defn}

\begin{remark}
    We exclude the maximal and the minimal elements from the lattices $\Co G$ and $\Sub G$ because in the converse case the classifying spaces will be contractible and this is not informative.
\end{remark}

Thus the group $G$ acts on $\Co G$ by conjugation and left shift with respect to the order. Also the action by conjugation induce an action on a segment $[N_1,N_2]$ including $\Sub G$ where $N_1$ and $N_2$ are normal subgroups of $G$.

Now consider the lattice $\Co G$ (or $\Sub G$) as a small category. An action of $G$ on it gives us an action on the classifying space $\B\Co G$ (or $\B\Sub G$). So the proposed method of spectral sequences allows us to reveal a relation between the cohomology of the group $G$ and the cohomology of the lattices $\Co G$ and $\Sub G$.

\begin{thm}
    Let $L$ be one of the lattices $\Co G$ or $\Sub G$ and let an action of $G$ on $L$ (by conjugation or shift) be fixed. If the classifying space $\B L$ is aspherical, then there exists a (co)homological spectral sequence convergent to the (co)homology of the group $S$:

    $$
        \begin{array}{lllll}
            E^2_{p,q} & = & H_p\bigl(G,H_q(L)\bigr) & \Rightarrow & H_{p+q}(S),\\
            E_2^{p,q} & = & H^p\bigl(G,H^q(L)\bigr) & \Rightarrow & H^{p+q}(S),\\
        \end{array}
    $$

    where $1 \to \pi_1(\B L) \to S \to G \to 1$ is the extension subordinated to the action of $G$ on $\B L$.
\end{thm}

\begin{remark}
    If the lattice $L$ is a wedge product of circles ($\pi_1(L)$ is a free non-abelian group) then $G$ satisfies the conditions of the previous theorem. Subgroup lattices of many minimal simple groups are homotopy equivalent to wedges of circles. This fact was proved by Shareshian (see \cite{Shareshian}). In this case the constructed spectral sequences are two-row.
\end{remark}

\begin{example}
    Let $G=\mathbb Z_{pq}$ where $p$ and $q$ are simple numbers and let $\Gamma_{pq}=\B\Co\mathbb Z_{pq}$ be the graph constructed by the lattice $\Co\mathbb Z_{pq}$. The shift action of $G$ is considered. The following extension is subordinated to this action:

    $$
        1 \to \pi_1(\Gamma_{pq}) \to S_{pq} \to \mathbb Z_{pq} \to 1.
    $$

    There are $pq+p+q$ vertexes and $2pq$ edges in $\Gamma_{pq}$. Hence $\Gamma_{pq}$ is homotopy equivalent to the wedge $\bigvee^{(p-1)(q-1)}S^1$ of $(p-1)(q-1)$ circles. Consequently, $H_0(\Gamma_{pq})=\mathbb Z$, $H_1(\Gamma_{pq})=\mathbb Z^{(p-1)(q-1)}$, $H_p(\Gamma_{pq})=0$ for all $p\geq 2$, and $\pi_1(\Gamma{pq}) = F^{(p-1)(q-1)}$ is a non-abelian free group with $(p-1)(q-1)$ generators.
    Let us calculate the cohomology of the group $S_{pq}$ by spectral sequence of Hochschild-Mostov: $H_k(\mathbb Z_{pq},H_0(\Gamma_{pq})) = H_k(\mathbb Z_{pq})$ as $\mathbb Z_{pq}$ acts trivially on $H_0(\Gamma_{pq})=\mathbb Z$. Also it is easy to construct explicit generators of the group $H_1(\Gamma_{pq})$ in the graph $\Gamma_{pq}$, calculate the explicit action matrix of a generator of the group $\mathbb Z_{pq}$ on $H_1(\mathbb Z_{pq}$, write out a necessary resolvent of $\mathbb Z_{pq}$ for the module $H_1(\Gamma_{pq})$ by this matrix and get $H_k(\mathbb Z_{pq},H_1(\Gamma_{pq}))=0$ for all $k$ \footnote{The exact proof of this fact is removed from the paper texts because the used ideas are very simple and precise calculations are very cumbersome.}.
    Thus we get $H_k(S_{pq})=H_k(\mathbb Z_{pq}$ by spectral sequence.
\end{example}

The author thanks M. Zelikin and I. Zhdanovkiy for many valuable discussions.


\begin{thebibliography}{99}

\bibitem{Bredon} {\it Glen E. Bredon} ``Introduction to compact transformation groups'', Academic Press, New York -- London, 1972
\bibitem{Bredon_Puchki} {\it Glen E. Bredon} ``Sheaf theory'', McGraw-Hill, New York, 1967.
\bibitem{Feigen_Fuks} {\it B. Fiegen, D. Fuks} ``Cohomologii group i algebr lie'', Itogi Nauki i tehniki, Ser. Sovr. Probl. Mat. Fund. Narp., 1988, tom 210, p 121-209 [in russian]
\bibitem{Husemoller} {\it D. Husemoller}, ``Fibre bundles'', McGraw-Hill Book Company, New York, St. Louis, San Francisco, Toronto, London, Sydney, 1966
\bibitem{Brown} {\it K.S. Brown} ``Cohomolgy of groups'', Springer-Vergal, New York, Heidelberg, Berlin, 1982.
\bibitem{EilenbergMacLane} {\it S. Eilenberg, S. MacLane }, ``Cohomology theory in abstract groups. II. Group extensions with a non-abelian kernel'', Ann. Math., 1947, (2) 48, p. 199-236.
\bibitem{Grothendieck} {\it A. Grothendieck}, ``Theorie de la descente, etc.'', Seminaire Bourbaki, 195 (1959-1960).
\bibitem{Hatcher} {\it A. Hatcher} ``Algebraic Topology'', Cambridge University Press, 2002.
\bibitem{Quillen_CS} {\it D. Quillen}, ``Higher algebraic K-theory'', Proceedings \ of \ the \ International Congress of Mathematicians (Vancouver, B. C., 1974), Vol. 1 (1975), Montreal, p. 77-139.
\bibitem{Segal_CS} {\it G. Segal}, ``Classifying spaces and spectral sequences'', Publications Mathématiques de l'IHÉS, 34 (1968), p. 105-112
\bibitem{Shareshian} {\it J. Shareshian}, ``On the shellability of the order complex of the subgroup lattice of a finite group'', Trans. Amer. Math. Soc. 353 (2001), 2689–2703.


\end{thebibliography}
\end{document}